\documentclass[reqno,11pt]{article}
\usepackage{amsfonts,amsmath,amsthm,amssymb,amscd,url}
\usepackage[all]{xy}
\theoremstyle{plain}
\newtheorem{theorem}{Theorem}
\newtheorem{proposition}[theorem]{Proposition}

\newtheorem{lemma}[theorem]{Lemma}

\theoremstyle{definition}

\theoremstyle{remark}
\newtheorem{remark}[theorem]{\upshape Remark}
\newcommand{\blue}[1]{#1}
\usepackage{color,url}

\newcommand{\R}{\mathbb R}
\newcommand{\C}{\mathbb C}

\newcommand{\Z}{\mathbb Z}
\newcommand{\zz}[2]{\frac{\partial^{#2}}{\partial z_{#1}^{#2}}}
\newcommand{\z}[1]{\frac{\partial}{\partial z_{#1}}}
\newcommand{\bz}{\overline{z}}
\newcommand{\bzz}[1]{\frac{\partial}{\partial \bz_{#1}}}
\newcommand{\D}{\mathcal{D}}
\newcommand{\I}{\mathcal{I}}
\newcommand{\M}{\mathcal{M}}
\newcommand{\B}{\mathcal{B}}
\newcommand{\h}{\mathfrak{h}}
\newcommand{\gl}{\mathfrak{gl}}

\newcommand{\End}{\operatorname{End}}
\newcommand{\Hom}{\operatorname{Hom}}

\title{A D-module approach to invariant distributions\\ 
with finitely many orbits}
\author{Hiroyuki Ochiai\thanks{ 
This work is partially supported by JSPS Grant-in-Aid for Transformative Research Areas (A) (22H05107).
}}
\begin{document}
\maketitle
\begin{abstract}
\blue{
Tauchi provides an example illustrating the action of a real algebraic subgroup $H$ of $GL(2n, \R)$ with finitely many orbits on $\R^{2n}$, while the dimension of the space of relative $H$-invariant distributions on $\R^{2n}$ is infinite. We offer a perspective on this example from the viewpoint of D-modules, where we explicitly determine the simple quotient regular holonomic D-modules and demonstrate that the distributions exhibit an enlarged symmetry.}
\end{abstract}


\section{Introduction}
In 2018, Tauchi \cite{T1} discovered a representation
of a Lie group with finitely many orbits, where 
the dimension of relatively invariant distributions with
a fixed character is infinite.
This intriguing representation is motivated by providing an example demonstrating 
that the smooth regular representation on some homogeneous space of $SL(2n,\R)$ with finitely many orbits under a real mirabolic subgroup has infinite multiplicity of a degenerate principal series representation \cite{T2}, \cite{T3}.
Distributions of interest are characterized by being annihilated by vector fields corresponding to a Lie algebra. In this paper, we explore the left ideal generated by these vector fields in the ring of differential operators and the corresponding D-modules. This D-module must be non-holonomic due to the infinite dimension of its solutions. We identify countably many quotient D-modules, each turning out to be a simple holonomic D-module. Subsequently, we revisit the Lie algebraic interpretation of these quotient D-modules. It's important to note that the group used here is solvable (upper triangular), potentially resulting in the corresponding D-module having irregular singularities. Nevertheless, these simple holonomic D-modules must exhibit regular singularities.

The interplay between Lie algebra and D-modules for relatively invariant distributions was leveraged by the author \cite{O1} to gain insights into the seminal work of van Dijk. In \cite{vD84}, among other contributions, he highlights the non-existence of a spherical distribution — an invariant eigen-distribution of the pseudo-Laplacian — on the tangent space of a rank-one semisimple symmetric space $SL(m+1, \mathbb{R})/GL(m,\mathbb{R})$ supported on the singular nilpotent orbits. This marked the beginning of harmonic analysis on this semisimple symmetric space \cite{vDP86, vdP90, vd08, vd09} and the exploration of a generalized Gelfand pair \cite{vd842, vD86, vdBook}.
Kowata \cite{Ko2} establishes the non-existence of singular invariant distributions and hyperfunctions \cite{Ko1} without resorting to the pseudo-Laplacian, subject to certain conditions. For instance, when considering the tangent space of $SL(m+1,\mathbb{R})/GL(m,\mathbb{R})$, Kowata identifies a larger symmetry, $SO_0(m,m)$, compared to $GL(m,\mathbb{R})$. The author \cite{O1} rearticulates this finding in terms of D-modules. Notably, in this tangent space setting, as well as the global setting \cite{O2}, two Lie algebras, $\mathfrak{gl}(m,\mathbb{R})$ and $\mathfrak{so}(m,m)$, exhibit the same D-modules. As a corollary, we deduce an automatic extension of the symmetry of spherical distribution and the non-existence of singular spherical distribution. In other words, there is no contribution of singular orbits under the smaller Lie group.

Tauchi's example focuses primarily on real projective spaces $\mathbb{P}^3$ and $\mathbb{P}^5$ of dimensions three and five, respectively. These spaces correspond to the real vector spaces $\mathbb{R}^4 \setminus \{0\}$ and $\mathbb{R}^6 \setminus \{0\}$ with the origin removed and divided by $\mathbb{R}^{\times}$. In our analysis, we operate on vector spaces $\mathbb{R}^{2n} = \mathbb{C}^n$ with $n=2,3$ without dividing by $\mathbb{R}^{\times}$, and we observe that including the origin poses no issues.

We introduce various real Lie subgroups of $GL(2n, \mathbb{R})$ and their Lie algebras, including their complexifications. Further, we delve into the orbit decomposition of $\mathbb{R}^{2n}$ and $\mathbb{C}^{2n}$ for each case.

In the context of the ring of differential operators with polynomial coefficients, denoted as 
\[\mathcal{D} = \mathbb{C}[z_1, \ldots, z_{2n},\blue{\z1, \ldots, \z{2n}]},\] we define the left ideal $\mathcal{I}$ of $\mathcal{D}$ and the corresponding D-module $\mathcal{D}/\mathcal{I}$, along with its hyperfunction or distribution solutions. The reference for D-modules will be \cite{HTT, K1}.

\section{$n=2$}
\subsection{Lie algebra}

\blue{We} define a solvable Lie subalgebra of $\gl(4,\C)$ by
\[
\h_1 :=
\mathfrak{a} \oplus \mathfrak{n}
:=
( \C E_{11} \oplus \C E_{22} \oplus \C E_{33} \oplus \C E_{44} )
\oplus ( \C E_{14} 
\oplus \C E_{24} \oplus \C E_{31} \oplus \C E_{32} \oplus \C E_{34} ).
\]
The Lie algebra $\h_1$ is illustrated as
\[
\begin{pmatrix}
* & 0 & 0 & * \\
0 & * & 0 & * \\
* & * & * & * \\
0 & 0 & 0 & * \\
\end{pmatrix}
\begin{pmatrix}
z_1\\
z_2\\
z_3\\
z_4
\end{pmatrix}
,
\qquad
\begin{pmatrix}
* & * & * & * \\
0 & * & 0 & * \\
0 & 0 & * & * \\
0 & 0 & 0 & *
\end{pmatrix}
\begin{pmatrix}
z_3\\
z_2\\
z_1\\
z_4
\end{pmatrix},
\]
where the figure on the left is represented in standard coordinates, whereas the figure on the right illustrates the upper-triangular realization of the Lie algebra. 
Put 
$c = \begin{pmatrix}
0 & 0 & 1 & 0 \\
0 & 1 & 0 & 0 \\
1 & 0 & 0 & 0 \\
0 & 0 & 0 & 1 \end{pmatrix}$
and a Lie algebra isomorphism
$\iota : \gl(4) \rightarrow \gl(4)$ by $\iota(A) = c A c^{-1}$.

We now consider the several \blue{subalgebras}
$\mathfrak{h}_3 \subset \mathfrak{h}_2 \subset \mathfrak{h}_1$ 
defined by
\begin{align*}
\mathfrak{h}_2 &:=
\mathfrak{a} \oplus ( \C E_{14} \oplus \C E_{32} ).
\\
\mathfrak{h}_3 &:= ( \C(E_{11}+E_{22}) \oplus \C(E_{33}+E_{44}))
\oplus ( \C E_{14} \oplus \C E_{32} ).
\end{align*}
\[
\h_1=\begin{pmatrix}
* & 0 & 0 & * \\
0 & * & 0 & * \\
* & * & * & * \\
0 & 0 & 0 & * \\
\end{pmatrix},
\quad
\h_2=\begin{pmatrix}
* & 0 & 0 & * \\
0 & * & 0 & 0 \\
0 & * & * & 0 \\
0 & 0 & 0 & * \\
\end{pmatrix},
\quad
\h_3=
\begin{pmatrix}
a_{11} & 0 & 0 & a_{14} \\
0 & a_{11} & 0 & 0 \\
0 & a_{32} & a_{44} & 0 \\
0 & 0 & 0 & a_{44} \\
\end{pmatrix}.
\]
\[
\iota(\h_1)=\begin{pmatrix}
* & * & * & * \\
0 & * & 0 & * \\
0 & 0 & * & * \\
0 & 0 & 0 & *
\end{pmatrix},
\quad
\iota(\h_2)=\begin{pmatrix}
* & * & 0 & 0 \\
0 & * & 0 & 0 \\
0 & 0 & * & * \\
0 & 0 & 0 & * \\
\end{pmatrix},
\quad
\iota(\h_3)=
\begin{pmatrix}
a_{44} & a_{32} & 0 & 0 \\
0 & a_{11} & 0 & 0 \\
0 & 0 & a_{11} & a_{14} \\
0 & 0 & 0 & a_{44}
\end{pmatrix}.
\]

\subsection{Real form}

Let $X:=\C^4$ be a complex four-dimensional vector space
with the standard coordinates $\mathbf{z}=(z_1, z_2, z_3, z_4 )^T$.
Let
$J:=\begin{pmatrix}
0 & 0 & 1 & 0 \\
0 & 0 & 0 & 1 \\
1 & 0 & 0 & 0 \\
0 & 1 & 0 & 0 
\end{pmatrix}$.
We define a real four-dimensional vector space
\[
X_\R:=\{ \mathbf{z} 
= (z_1, z_2, z_3, z_4 )^T \in \C^4 \mid \mathbf{z} = J \overline{\mathbf{z}} \}
= \{(z_1,z_2, \overline{z_1}, \overline{z_2})\blue{^T}\mid z_1,z_2 \in \C \}.
\]
Then $X_\R$ is a real form of the complex vector space $X=\C^4$.
The corresponding real form of the complex Lie algebra $\mathfrak{gl}(4,\C) = \blue{\End}_\C(X)$ 
is 
$\mathfrak{g}_\R:=\{ X \in \mathfrak{gl}(4,\C) \mid J \overline{X} J = X \}$.

We define the real Lie algebra $\h_{i\R} := \h_i \cap \mathfrak{g}_\R$
for $i=\blue{1},2,3$.
\begin{lemma}
\begin{itemize}
\item
$\mathfrak{g}_\R=\left\{ \begin{pmatrix}
a_{11} & a_{12} & a_{13} & a_{14} \\
a_{21} & a_{22} & a_{23} & a_{24} \\
\overline{a_{13}} & \overline{a_{14}} & \overline{a_{11}} & \overline{a_{12}} \\
\overline{a_{23}} & \overline{a_{24}} & \overline{a_{21}} & \overline{a_{22}} 
\end{pmatrix} \right\} $.
\item
$\h_{2\R}
=\left\{\begin{pmatrix}
a_{11} & 0 & 0 & a_{14} \\
0 & a_{22} & 0 & 0 \\
0 & \overline{a_{14}} & \overline{a_{11}} & 0 \\
0 & 0 & 0 & \overline{a_{22}} 
\end{pmatrix}\right\}
=
\left\{\begin{pmatrix}
a_{11} & 0 & 0 & a_{14} \\
0 & \overline{a_{44}} & 0 & 0 \\
0 & \overline{a_{14}} & \overline{a_{11}} & 0 \\
0 & 0 & 0 & a_{44}
\end{pmatrix}\right\}
$, which is a real form of $\h_2$.
\item
$\h_{1\R}=\h_{2\R}$,
\blue{which is not a real form of $\h_1$.}

\item
$\blue{\h_{3\R}}
= \left\{\begin{pmatrix}
a_{11} & 0 & 0 & a_{14} \\
0 & a_{11} & 0 & 0 \\
0 & \overline{a_{14}} & \overline{a_{11}} & 0 \\
0 & 0 & 0 & \overline{a_{11}}
\end{pmatrix}\right\}$,
which is \blue{a} real form of $\h_3$.
\end{itemize}
\end{lemma}
\blue{The image of each real Lie algebra by $\iota$ is given as}
\begin{align*}
\iota(\mathfrak{g}_\R)
&=\left\{\begin{pmatrix}
\overline{a_{11}} & \overline{a_{14}} & \overline{a_{13}} & \overline{a_{12}} \\
a_{23} & a_{22} & a_{21} & a_{24} \\
a_{13} & a_{12} & a_{11} & a_{14} \\
\overline{a_{21}} & \overline{a_{24}} & \overline{a_{23}} & \overline{a_{22}} 
\end{pmatrix}\right\},\\
\iota(\h_{2\R})
&= \left\{\begin{pmatrix}
\overline{a_{11}} & \overline{a_{14}} & 0 & 0 \\
0 & \overline{a_{44}} & 0 & 0 \\
0 & 0 & a_{11} & a_{14} \\
0 & 0 & 0 & a_{44}
\end{pmatrix}\right\},
\quad
\iota(\h_{3\R})
=\left\{\begin{pmatrix}
\overline{a_{11}} & \overline{a_{14}} & 0 & 0 \\
0 & a_{11} & 0 & 0 \\
0 & 0 & a_{11} & a_{14} \\
0 & 0 & 0 & \overline{a_{11}}
\end{pmatrix}\right\}\blue{.}
\end{align*}
The dimension of each Lie algebra is
\[
\dim_\C \h_1 = 9,\ 
\dim_\C \h_2 = 6,\ 
\dim_\C \h_3 = 4,\ 
\blue{\dim_\R \h_{1\R}=} \dim_\R \h_{2\R}=6,\ 
\dim_\R \h_{3\R}=4.
\]

\subsection{Orbit decomposition}
For each Lie algebra or its corresponding Lie group,
we consider the orbit decomposition by 
the action of $GL(4,\C) \supset H_1 \supset H_2 \supset H_3$ on $X=\C^4$
and the action of $G_\R\supset H_{1\R}= H_{2\R} \supset H_{3\R}$ on $X_\R$.
The upper triangular realization $\iota(\h_i)$ is preferred over the standard realization $\h_i$ to facilitate the observation of the orbit decomposition.

\subsubsection{\blue{C}omplex}
We define Zariski locally closed subsets $O_2,O_1,O_0$ of $\C^2=\{(z_3,z_2)^\blue{T} \in \C^2\}$
and $O'_2, O'_1,O'_0$ of $\C^2=\{(z_1,z_4)^\blue{T} \in \C^2\}$ by
\begin{align*}
O_2 &:=\{\blue{(z_3,z_2)^\blue{T} \in \C^2 \mid} z_2 \neq 0\},\\
O_1&:=\{\blue{(z_3,z_2)^\blue{T} \in \C^2 \mid} z_2=0, z_3 \neq 0\},\\
O_0 &:=\{\blue{(z_3,z_2)^\blue{T} \in \C^2 \mid} z_2=z_3=0 \},
\\
O'_2 &:=\{\blue{(z_1,z_4)^\blue{T} \in \C^2 \mid} z_4 \neq 0\},\\
O'_1&:=\{\blue{(z_1,z_4)^\blue{T} \in \C^2 \mid} z_4=0, z_1 \neq 0\},\\
O'_0 &:=\{\blue{(z_1,z_4)^\blue{T} \in \C^2 \mid} z_4=z_1=0 \}.
\end{align*}
Here, the index denotes the dimension of each orbit.
We define $O_{ij} = O_i \times O'_j \subset \blue{\C^2\times \C^2 \subset} \C^4$
for $i,j=0,1,2$.
\blue{Then \[\displaystyle \C^2 = \bigsqcup_{i=0,1,2} O_i,
\C^2 = \bigsqcup_{j=0,1,2} O'_j,
\C^4 = \bigsqcup_{i,j=0,1,2} O_{ij}.\]}
For $c \in \C^\times$,
we define (Tauchi)
\begin{align*}
O_{12c} &:= \{ (z_3,z_2,z_1,z_4) \in O_{12} \mid z_3 = c z_4 \},
\\
O_{21c} &:= \{ (z_3,z_2,z_1,z_4) \in O_{21} \mid z_1 = c z_2 \}.
\end{align*}
Note that 
$\displaystyle\blue{\bigsqcup_{c \in \C^\times}} O_{12c} = O_{12}$,
and
$\displaystyle\blue{\bigsqcup_{c \in \C^\times}} O_{21c} = O_{21}$.
\begin{lemma}
\begin{itemize}
\item $O_{ij}$ are $\blue{\iota(}H_2)$-orbits.
\item $O_{ij}$ remain $\blue{\iota(}H_3)$-orbits unless $(i,j)=(1,2),(2,1)$.
\item $O_{12c}$ and $O_{21c}$ are $\blue{\iota(}H_3)$-orbits \blue{ for $c \in \C^\times$}.
\item $O_{22}\cup O_{\blue{12}} \cup O_{\blue{02}}$,
$O_{\blue{21}}$, $O_{\blue{20}}$, $O_{11} \cup \blue{O_{01}}$, $\blue{O_{10}}$, $O_{00}$
are $\blue{\iota(}H_1)$-orbits.
\end{itemize}
\end{lemma}

\subsubsection{\blue{R}eal}
The real form $X_\R\blue{=\{(z_1,z_2,\overline{z_1},\overline{z_2}) \mid z_1, z_2 \in \C\}}$ is a subset of $O_{22} \cup O_{11} \cup O_{00}$.

\begin{lemma}
\begin{itemize}
\item
There are three $H_{3\R}$-orbits
\begin{align*}
O_{22\R}&:=\{(z_1,z_2,\blue{\overline{z_1}, \overline{z_2}}) \in X_{\R}\mid z_2 \neq 0\},\\
O_{11\R}&:=\{(z_1,z_2,\blue{\overline{z_1}, \overline{z_2}}) \in X_{\R}\mid z_2=0, z_1 \neq 0\},\\ 
O_{00\R}&:=\{(z_1,z_2,\blue{\overline{z_1}, \overline{z_2}}) \in X_{\R}\mid z_2=z_1=0 \}
\end{align*}
on $X_\R$.
\item
Each orbit under $H_{3\R}$ is also 
an orbit under $H_{2\R}$.
\item
The real manifold
$O_{ii\R}$ is a real form of the complex manifold $O_{ii}$
for $i=0,1,2$.
\end{itemize}

\end{lemma}
\begin{remark}
The groups $H_3$ and $H_{3\R}$ 
and their orbit decompositions are given by Tauchi \cite{T1}.
{\blue It is important that $H_3$ has infinitely many orbits in $\C^4$ and
$H_{3\R}$ has finitely many orbits in the real form $X_\R$.}
\end{remark}

\subsection{D-module}
In $X=\C^4$, 
we use coordinates $(z_1,z_2,z_3,z_4) \in X$,
and $(z_1,z_2,z_3,z_4,\zeta_1,\zeta_2,\zeta_3,\zeta_4) \in T^*X$
for the cotangent bundle.
The ring of differential operators on $X$ is denoted by $\D=\D_X=\C[z_1,z_2,z_3,z_4, \z{1},\z 2, \z 3,\z{4}]$.
For a complex subspace $Y:=\{z_2=z_4=0\}$, which is the closure of $O_{11}$,
we define a D-module
\[
\mathcal{B}_{Y|X}
= \mathcal{H}^2_{[Y]}(\mathcal{O}_X)
= \D/(\D \z 1 + \D z_2 + \D \z 3 + \D z_4).
\]
The characteristic variety of $\mathcal{B}_{Y|X}$ is the conormal bundle
$\{ \zeta_1= z_2= \zeta_3= z_4=0 \} = T^*_Y X \subset T^* X$ of $Y$,
and the multiplicity of $T^*_Y X$ is one,
that is, $\mathcal{B}_{Y|X}$ is a simple holonomic system.
We denote by $\delta(z_2,z_4) \in \mathcal{B}_{Y|X}$ the class represented by $1\in \D$.
For $l \in\Z_{\ge0}$, we consider the element
$T_l:=(z_1 \z 2)^l \delta(z_2,z_4) \in \mathcal{B}_{Y|X}$.
Then 
\begin{lemma}\label{lemma:isom}
The annihilator $\I_{1l}$ of $T_l$ is the left ideal of $\D$ generated by
\begin{equation}\label{I1}
z_1 \z 1 -l, \ 
\zz 1{l+1}, \ 
z_2 \z 2 + l+1, \ 
z_2^{l+1},\ 
\z 3, \ 
z_4.
\end{equation}
\end{lemma}
\blue{
\begin{proof}
The partial algebraic Fourier transform 
\[
z_1 \mapsto z_1, z_2 \mapsto \frac\partial{\partial z_2}, 
z_3 \mapsto z_3, z_4 \mapsto \frac\partial{\partial z_4}, 
\frac\partial{\partial z_1} \mapsto \frac\partial{\partial z_1},
\frac\partial{\partial z_2} \mapsto -z_2,
\frac\partial{\partial z_3} \mapsto \frac\partial{\partial z_3},
\frac\partial{\partial z_4} \mapsto -z_4\]
gives an algebra isomorphism from $\D$ to $\D$.
This isomorphism sends $\mathcal{B}_{Y|X}$ to
\[
\D/(\D\frac\partial{\partial z_1}
+ \D\frac\partial{\partial z_2}+ 
\D\frac\partial{\partial z_3}+ 
\D\frac\partial{\partial z_4}) \simeq \mathcal{O}_X.
\]
The image of $\delta(z_2,z_4)$ is $1 \in \mathcal{O}_X$
and the image of $T_l$ is $(-z_1 z_2)^l \in \mathcal{O}_X$.
In the one variable case,
the annihilator of $z_1^l$ in $\C[z_1,\frac\partial{\partial z_1}]$
is the left ideal of $\C[z_1,\frac\partial{\partial z_1}]$
generated by $z_1 \frac\partial{\partial z_1} -l$ and $\frac{\partial^{l+1}}{\partial z_1^{l+1}}$.
The inverse partial Fourier transform of 
$z_2 \frac\partial{\partial z_2} -l$ is
$\frac\partial{\partial z_2} (-z_2) - l = -\left(z_2 \frac\partial{\partial z_2} + l +1\right)$,
which
demonstrates the desired lemma.
\end{proof}}
Let 
$\I_3$ be the left ideal generated by
\begin{equation}
z_1 \z 1 + z_2 \z 2+1, \ 
\z 3, \ 
z_4.
\end{equation}
We define left D-modules $\M_{1l} := \D/\I_{1l}$
and $\M_3:=\D/\I_3$.

\begin{theorem}\label{theorem 6}
\begin{itemize}
\item[(i)]
We have a D-module isomorphism
$\displaystyle
\M_{1l} \ni P \mapsto P\blue{\circ} (z_1 \z 2)^l \blue{\delta(z_2,z_4)} \in \mathcal{B}_{Y|X}
$.
\blue{Especially, $\M_{1l}$ is a simple holonomic system.}
\item[(ii)]
The natural D-module homomorphism
$\displaystyle\M_3 \rightarrow \M_{1l}$
is surjective.
\item[(iii)]
\blue{
For any positive integer $L$, the natural D-module homomorphism
$\displaystyle\M_3 \rightarrow \bigoplus_{l=0}^L \M_{1l}$
is surjective.}
\end{itemize}
\end{theorem}
\begin{proof}
\blue{
(i) follows from Lemma~{lemma:isom}, and
(ii) follows from the fact that $\I_3 \subset \I_{1l}$ for all $l \in \Z_{\ge 0}$.
We prove (iii).
For a given $P_l \in \D$ with $l=0,1,\ldots,L$,
we put
\[
P := \sum_{l=0}^L P_l \prod_{0\leq k \leq L, k \neq l} 
\frac{z_1 \frac{\partial}{\partial z_1}-k}{l-k} \in \D,
\]
then $P \equiv P_l \mod \I_{1l}$ with $l=0,1,\ldots,L$ follows.}
\end{proof}

We observe that
\begin{equation}\label{I3}
z_4 \z 1,\ 
z_2 \z 3,\ 
z_1 \z 1+z_2 \z 2+1,\ 
z_3 \z 3+z_4 \z 4+1 \in \I_3.
\end{equation}
We will interpret these elements as images of a Lie algebra.
Let 
\[
\rho:\gl(4)=\gl(4,\C) \ni E_{ij} \mapsto -z_j \z i \in \D
\]
be a Lie algebra homomorphism
induced from the standard representation of $GL(4,\C)$ on $X=\C^4$.
We define the character $\chi_l: \h_1=\mathfrak{a} \oplus \mathfrak{n} \to \C$
by 
\[
\chi_l(E_{11}) =-l,\ 
\chi_l(E_{22}) = l+1,\ 
\chi_l(E_{33})=0,\ 
\chi_l(E_{44})=1
\]
and $\chi_l(\mathfrak{n})=0$.
The restriction of $\chi_l$ on \blue{$\h_3$} is denoted by $\chi$
since it does not depend on $l$.
Explicitly,
\[
\chi(E_{11}+E_{22})=1,\ 
\chi(E_{33}+E_{44})=1.
\]
Then Lemma~\ref{lemma:isom} and the equation~\eqref{I3} \blue{show} that 
\begin{equation}\label{diagram}
\begin{array}{ccc}
(\rho-\chi)(\h_3) &\subset& \I_3 \\
\cap & & \cap \\
(\rho-\chi_l)(\h_1) &\subset& \I_{1l}
\end{array}
\end{equation}

\begin{remark}
The Lie algebra homomorphism $\rho$ induces the associative algebra homomorphism
\[
\tilde{\rho}: U(\gl(4)) \rightarrow 
\{P \in \D \mid [P, z_1 \z 1+z_2 \z 2+ z_3 \z 3+z_4 \z 4]=0 \}
\] from the universal enveloping algebra 
of the Lie algebra $\gl(4)$ onto the subalgebra of $\D$ of homogeneous degree zero.
Note that the map $\tilde{\rho}$ is not injective;
for example, a non-zero element $E_{14} E_{32} - E_{12} E_{34} \in U(\gl(4))$
belongs to the kernel of $\tilde{\rho}$.
Therefore, it is possible that two different Lie subalgebras of $\gl(4)$
generate two distinct left ideals of $U(\gl(4))$, whose images under $\tilde{\rho}$ may coincide.
This is one of the reasons why invariance under a smaller Lie algebra extends to that under a larger Lie algebra.
\end{remark}

\subsection{Relatively invariant distribution}
Let $\B_{X_\R}$ be sections of hyperfunctions on $X_\R$.
From the diagram \blue{\eqref{diagram}},
we conclude the following on solutions of differential equations.
\begin{theorem}
\[
\begin{array}{ccc}
\{ u \in \B_{X_\R} \mid P u = 0, \forall P \in (\rho-\chi)(\h_3) \} 
& \supset &
\{ u \in \B_{X_\R} \mid P u = 0, \forall P \in \I_3 \}
\\
\cup && \cup 
\\
\displaystyle\bigoplus_{l \in \Z_{\ge0}}
\{ u \in \B_{X_\R} \mid P u = 0, \forall P \in (\rho-\chi_l)(\h_1) \} 
& \supset &
\displaystyle\bigoplus_{l \in \Z_{\ge0}}
\{ u \in \B_{X_\R} \mid P u = 0, \forall P \in \I_{1l} \}
\end{array}
\]
\end{theorem}

Tauchi defines \cite{T1} a distribution
\begin{equation}
T_2^l = z_1^l \zz 2l \delta(z_2,\bz_2)
=(z_1 \z 2)^l \delta(z_2,\bz_2)
\end{equation}
on the real manifold $X_\R\blue{=\{(z_1,z_2,\overline{z_1},\overline{z_2}) \mid z_1, z_2 \in \C\}\simeq \R^4}$.
\blue{Here $\delta(z_2,\bz_2)$ denotes the Dirac distribution $\frac{1}{-2\sqrt{-1}}\delta(\Re z_2, \Im z_2)$ on $\R^4$.}
The support of $T_2^l$ is the closure of $O_{11\R}$, \blue{that is, $O_{11\R} \sqcup O_{00\R}$}.
\begin{proposition}
$T_2^l \in 
\{ u \in \B_{X_\R} \mid P u = 0, \forall P \in \I_{1l} \}
\blue{\simeq\Hom}_{\D}(\M_{1l}, \B_{X_\R})$.
Especially, $T_2^l$ is relatively invariant under $\h_1$ with the character $\chi_l$,
and is relatively invariant under $\h_3$ with the character $\chi$.
\end{proposition}

Note that the relative invariance of $T_2^l$ under $\h_3$ is due to Tauchi.
We regard Theorem~\ref{theorem 6} as a D-module counterpart of this result.

\section{$n=3$}
\subsection{D-module}
In $X=\C^6$, 
we use coordinates $(z_1,z_2,z_3,z_4,z_5,z_6) \in X$.
The ring of differential operators on $X$ is denoted by $\D=\D_X=\C[z_1,z_2,z_3,z_4,z_5,z_6, \z{1},\z 2, \z 3,\z{4},\z 5, \z 6]$.
Let $Y=\{z_3=z_6=0\} \subset X=\C^6$,
and define a simple holonomic D-module $\B_{Y|X}$ by
\[
\mathcal{B}_{Y|X}
= \mathcal{H}^2_{[Y]}(\mathcal{O}_X)
= \D/(\D \z 1 + \D \z 2 + \D z_3 + \D \z 4 + \D \z 5 + \D z_6).
\]
The standard generator represented by $1 \in \D$ is denoted by
$\delta(z_3,z_6) \in \B_{Y|X}$.
We define a section of $\B_{Y|X}$.
\begin{align}
T_l 
&= (z_2\z 3+z_4 \z 5)^{l} z_2 z_5 \delta(z_3, z_6)
\\
&= z_2^l \left( z_2 z_5 \zz 3l + l z_4 \zz 3{l-1}\right)\delta(z_3, z_6)
\\
&= ( z_2 z_5  - z_3 z_4 ) (z_2\z 3)^{l}\delta(z_3, z_6).
\notag
\end{align}

\begin{lemma}\label{n=3lemma}
The annihilator ideal $\I_{1l}$ of 
$T_l$ is generated by the following elements:
\begin{eqnarray*}
&\displaystyle \z 1,\ 
z_6,\ 
z_2 \z 2+z_3 \z 3,\ 
z_2 \z 2+z_4 \z 4-1-l,\ 
z_2 \z 4+z_3 \z 5,
\\
&\displaystyle \blue{\zz 2{l+1} \z 4,\ 
z_3^l \z 4,\ 
\zz 2{l+2} \z 5,\ 
z_3^{l+1} \z 5,}
\\
&\displaystyle  \zz 42,\ 
\zz 52,\ 
\frac{\partial^2}{\partial z_4 \partial z_5},\ 
z_4 \z 4+z_5 \z 5-1.
\end{eqnarray*}
\end{lemma}
\begin{proof}
We define a left ideal $\I'_{l}$ of $\D$ 
generated by
\begin{eqnarray*}
\displaystyle 
\z 1, \ 
\z 4, \ 
\z 5, \ 
z_6,\  
z_2 \z 2-\blue{l},\ 
\zz 2{l+\blue{1}},\ 
z_3 \z 3+l,\ 
z_3^l. 
\end{eqnarray*}
\blue{Using a similar argument to Theorem~\ref{theorem 6}(i)},
we see that $\M'_{l} := \D/\I'_{l}$ forms a simple holonomic system.
Let $\I''$ be the left ideal generated by elements in the statement of Lemma
other than $z_2 \z 4+z_3 \z 5$.
\blue{
By computing the image of generators of left ideals, we obtain
the property:
\begin{align}
\label{contain1}
\I'' z_4 \subset \I'_l &\mbox{ and } \I'' z_5 \subset \I'_{l+1},
\\
\I'_l \z 4 \subset \I''  &\mbox{ and }
\I'_{l+1} \z 5 \subset \I''.
\label{contain2}
\end{align}}
For example, $(z_2 \z 2-l) \z 4 = \z 4(z_2 \z 2+ z_4 \z 4-l-1)- z_4 \zz 42 \in \I''$.
We \blue{will} see that the map $\D \blue{\ni} P  \mapsto  (P z_4, P z_5) \in \D^{\oplus 2}$
induces a D-module \blue{isomorphism}
\[
\blue{\phi}:\D/\I'' \ni P \mapsto (P z_4, P z_5) \in \M'_{\blue{l}} \oplus \M'_{\blue{l+1}}.
\]
\blue{Actually, $\phi$ is well-defined from the property \eqref{contain1}.
For a given $P_1, P_2 \in D$,
put $P:= P_1 \frac{\partial}{\partial z_4} + P_2 \frac{\partial}{\partial z_4} \in \D$,
then $P z_4 - P_1 \in \I'_l$ and $P z_5 - P_2 \in \I'_{l+1}$.
This proves the surjectivity of $\phi$.
For the injectivity of $\phi$,
take a $P \in \D$ such that $(P \mod \I'')$ is in the kernel of $\phi$.
We may write 
$P=Q_1 \frac{\partial}{\partial z_4} + Q_2 \frac{\partial}{\partial z_5} + Q_3$
with $Q_i \in \C[z_1,z_2,z_3,z_4,z_5,z_6,\frac{\partial}{\partial z_1},\frac{\partial}{\partial z_2},
\frac{\partial}{\partial z_3},\frac{\partial}{\partial z_6}]$.
Then by assumption,
$0 \equiv P z_4 \equiv Q_1+Q_3 z_4 \mod \I'_l$ and
$0 \equiv P z_5 \equiv Q_2+Q_3 z_5 \mod \I'_{l+1}$.
There exist $Q_4 \in \I'_l, Q_5 \in \I'_{l+1}$ such that
$Q_1 = - Q_3 z_4 + Q_4$,
$Q_2 = -Q_3 z_5 + Q_5$.
Then $P = -Q_3 (z_4 \z 4 + z_5 \z 5 -1) 
+ Q_4 \frac{\partial}{\partial z_4} + Q_5 \frac{\partial}{\partial z_5} \in \I''$
by the property~\eqref{contain2}.
}

Note that the image of \blue{the map $\phi$} is the sum of two simple holonomic systems.
Therefore, $\D/(\I''+ D (z_2 \z 4+z_3 \z 5))$ is a simple holonomic system \blue{or zero}.
It is easy to see that the annihilator contains $\I''+ D (z_2 \z 4+z_3 \z 5)$;
then, maximality implies equality.
\end{proof}
We define a complex Lie subalgebra $\h_1 \subset \gl(6,\C)$ by
\[
\h_1 = \bigoplus_{j=1}^6 \C E_{1j} \oplus \bigoplus_{i=2}^6 \C E_{i6} \oplus
\C (E_{22}+E_{33}) \oplus \C(E_{44} + E_{55}) \oplus \C(E_{22}+ E_{44}) \oplus \C(E_{42} + E_{53}).
\]
and its character $\chi_l: \h_1 \rightarrow \C$
by
\begin{align*}
&\chi_l(E_{22}+E_{33})=0,\ 
\chi_l(E_{22}+E_{44})=-1-l,\ 
\chi_l(E_{44}+E_{55})=-1,\\
&\blue{\chi_l(E_{1j})=\chi_l(E_{i6})=\chi_l(E_{42}+E_{53})=0}.
\end{align*}
Then 
in terms of a Lie algebra homomorphism
\[
\rho: \gl(6,\C) \ni E_{ij} \mapsto - z_j \z i \in \D,
\]
Lemma~\ref{n=3lemma} \blue{implies} $(\rho-\chi_l)(\h_1) \subset \I_{1l}$.

We define a Lie subalgebra $\h_3\subset \h_1$ by
\[
\h_3 = \C(E_{11}+E_{22}+E_{33}) \oplus
\C(E_{44} + E_{55} + E_{66}) \oplus
\C(E_{42} + E_{53}) \oplus
\C(E_{15} + E_{26}) \oplus
\C E_{13} \oplus \C E_{46}.
\]
We denote by $\chi$ the restriction of $\chi_l$ on $\h_3$
since the restriction does not depend on $l$.
We define the left ideal $\I_3 \subset \D$ generated by
\blue{$(\rho-\chi)(\h_3)$, that is,}
\[
z_1 \z 1+z_2 \z 2+z_3 \z 3,
z_4 \z 4+z_5 \z 5 + z_6 \z 6 \blue{-1},
z_5 \z 1+z_6 \z 2,
z_2 \z 4+z_3 \z 5,
z_3 \z 1, z_6 \z 4.
\]
Then $(\rho-\chi)(\h_3) \subset (\rho-\chi_l)(\h_{1l})$ for all $l \in \Z_{\ge0}$,
and $(\rho-\chi)(\h_3) \subset \I_{3} \subset \I_{1l}$.
We have the same diagram \eqref{diagram}.

Now we move on to solutions of these D-modules.
We choose a real form $X_\R=\{(z_1,z_2,z_3, \overline{z_1}, \overline{z_2},\overline{z_3}) \mid (z_1,z_2,z_3) \in \C^3\}$ of $X=\C^6$.
Denote by $\B_{X_\R}$ the set of hyperfunctions on $X_\R$.
\begin{lemma}
Tauchi's distribution
$T_0^l = (\bz_{1} \bzz{2} + z_{2} \z{3})^l |z_{2}|^{2} \delta(z_3,\bz_3) \in \B_{X_\R}$ is a solution of $\D/\I_{1l}$.
\end{lemma}
This shows $T_0^l$ is \blue{relatively} invariant under
\begin{align*}
\h_{3\R}
&=\h_3 \cap \mathfrak{g}_\R \\
&= \blue{\{ 
a(E_{11}+E_{22}+E_{33}) +
\overline{a} (E_{44} + E_{55} + E_{66}) +
b(E_{42} + E_{53}) +
\overline{b} (E_{15} + E_{26}) +
c E_{13} + \overline{c}  E_{46} \mid a,b,c \in \C \}}
\end{align*}
\blue{with the character $\chi$}.
Orbit decomposition is similar to $n=2$, so we omit its description.

\bigskip
\blue{
We would like to express our sincere gratitude to the referee for their valuable comments and suggestions, which have greatly improved the quality and clarity of this paper.
}


\noindent
Hiroyuki Ochiai\\
Institute of Mathematics for Industry, Kyushu University\\
744 Motooka Fukuoka 819-0395 Japan\\
ochiai@imi.kyushu-u.ac.jp


\begin{thebibliography}{99}

\bibitem{HTT}{}
Ryoshi Hotta, Kiyoshi Takeuchi and Toshiyuki Tanisaki,
D-modules, perverse sheaves, and representation theory,
Progr. Math., 236
Birkh\" auser Boston, Inc., Boston, MA, 2008, xii+407 pp.

\bibitem{K1}{}
Masaki Kashiwara, 
D-modules and microlocal calculus, translated from the 2000 Japanese original by Mutsumi Saito, Transl. Math. Monogr., 217, Amer. Math. Soc., Providence, RI, 2003.

\bibitem{KvD}{}
M. T. Kosters and G. van Dijk,
Spherical distributions on the pseudo-Riemannian space $SL(n,\R)/GL(n-1,\R)$, 
J. Funct. Anal. 68 (1986), no. 2, 168--213.

\bibitem{Ko1}{}
Atsutaka Kowata,
On the construction of spherical hyperfunctions on $\R^{p+q}$,
Hiroshima Math. J.
21 (1991), 301--334.

\bibitem{Ko2}
Atsutaka Kowata,
Spherical hyperfunctions on the tangent space of
symmetric spaces,
Hiroshima Math. J.
21 (1991), 401--418


\bibitem{LS}
Thierry Levasseur and J. Toby Stafford,
Invariant differential operators on the tangent
space of some symmetric spaces,
Annales de l'institut Fourier, tome 49, no 6 (1999), p. 1711--1741.



\bibitem{O1}
Hiroyuki Ochiai, 
Invariant functions on the tangent space of a rank one semisimple symmetric space,
J. Fac. Sci. Univ. Tokyo Sect. IA Math.39(1992), no.1, 17--31.

\bibitem{O2}{}
Hiroyuki Ochiai, 
Invariant distributions on a non-isotropic pseudo-Riemannian symmetric space of rank one,
Indag. Math. (N.S.) 16 (2005), no. 3-4, 631--638.


\bibitem{T}
Toshiyuki Tanisaki, 
D-modules and representation theory.(English summary)Lie theory and representation theory, 177--219.
Surv. Mod. Math., 2
International Press, Somerville, MA, 2012.



\bibitem{T1}{}
Taito Tauchi, 
 Dimension of the space of intertwining operators from degenerate principal series representations,
Selecta Math. (N.S.) 24 (2018), no. 4, 3649--3662.

\bibitem{T2}{}
Taito Tauchi, 
Relationship between orbit decomposition on the flag varieties and multiplicities of induced representations,
Proc. Japan Acad. Ser. A Math. Sci. 95 (2019), no. 7, 75--79.

\bibitem{T3}{}
Taito Tauchi, 
Multiplicity of a degenerate principal series for homogeneous spaces with infinite orbits,
Proc. Amer. Math. Soc. 150 (2022), no. 2, 849--856.




\bibitem{vD84}{}
Gerrit van Dijk,
Invariant eigendistributions on the tangent space of a rank one semisimple symmetric space,
Math. Ann. 268 (1984), no. 3, 405--416.

\bibitem{vd842}{}
Gerrit van Dijk,
On generalized Gel'fand pairs,
Proc. Japan Acad. Ser. A Math. Sci. 60 (1984), no. 1, 30--34.

\bibitem{vD86}{}
Gerrit van Dijk,
On a class of generalized Gel'fand pairs,
Math. Z. 193 (1986), no. 4, 581--593.

\bibitem{vd862}{}
Gerrit van Dijk,
Harmonic analysis on rank one symmetric spaces,
Lecture Notes in Phys., 261
Springer-Verlag, Berlin, 1986, 244--252.


\bibitem{vdCOE}{}
Gerrit van Dijk,
Gelfand pairs and beyond,
COE Lect. Note, 11,
Math-for-Ind. (MI) Lect. Note Ser.
Kyushu University, Faculty of Mathematics, Fukuoka, 2008, ii+60 pp.

\bibitem{vd08}{}
Gerrit van Dijk,
$(GL(n+1,\R), GL(n,\R))$  is a generalized Gelfand pair,
Russ. J. Math. Phys. 15 (2008), no. 4, 548--551.



\bibitem{vd09}{}
Gerrit van Dijk,
$(U(p,q),U(p-1,q))$ is a generalized Gelfand pair,
Math. Z. 261 (2009), no. 3, 525--529.

\bibitem{vdBook}{}
Gerrit van Dijk,
Introduction to Harmonic Analysis and Generalized Gelfand Pairs,
Walter de Gruyter, Berlin  New York, 2009.

\bibitem{vDP86}{}
G. van Dijk and M. Poel,
The Plancherel formula for the pseudo-Riemannian space  $SL(n,\R)/GL(n-1,\R)$ ,
Compositio Math. 58 (1986), no. 3, 371--397.

\bibitem{vdP90}{}
G. van Dijk and M. Poel,
The irreducible unitary  $GL(n-1,\R)$-spherical representations of  $SL(n,\R)$,
Compositio Math. 73 (1990), no. 1, 1--30.


\end{thebibliography}
\end{document}